\documentclass[numbers,webpdf,imaiai]{ima-authoring-template}%

\graphicspath{{Fig/}}
\def \R {\mathbb{R}}
\def \E {\mathbb{E}}
\def \L {\mathbb L}
 \def \H {\mathbb H}
  \def\N{{\mathbb N}}
 
\usepackage{empheq}
\theoremstyle{thmstyletwo}%
\newtheorem{theorem}{Theorem}%  meant for continuous numbers
\newtheorem{proposition}[theorem]{Proposition}%
\newtheorem{remark}{Remark}%

\newtheorem{lem}{Lemma}
\numberwithin{equation}{section}

\begin{document}

\DOI{DOI HERE}
\copyrightyear{2021}
\vol{00}
\pubyear{2021}
\access{Advance Access Publication Date: Day Month Year}
\appnotes{Paper}
\copyrightstatement{Published by Oxford University Press on behalf of the Institute of Mathematics and its Applications. All rights reserved.}
\firstpage{1}

%\subtitle{Subject Section}

\title[Internal null controllability for the 1D-heat equation with DBC]{Internal null controllability for the one-dimensional heat equation with dynamic boundary conditions}

\author{El Mustapha Ait Ben Hassi, Mariem Jakhoukh, and Lahcen Maniar
\address{\orgdiv{Department of mathematics}, \orgname{Cadi Ayyad University, Faculty of Sciences Semlalia, LMDP-UMISCO}, \orgaddress{\street{B.P. 2390}, \postcode{40000}, \state{Marrakesh}, \country{Morocco}}}}
\author{Walid Zouhair* \ORCID{0000-0002-5939-2737}
\address{\orgdiv{Department of Mathematics}, \orgname{Ibn Zohr University, Faculty of Applied Sciences Ait Melloul}, \orgaddress{ \postcode{B.P. 6146}, \state{Agadir}, \country{Morocco}}}
\address{\it \Large Dedicated to the memory of Professor Hammadi Bouslous}}

\corresp[*]{Corresponding author: \href{walid.zouhair.fssm@gmail.com}{walid.zouhair.fssm@gmail.com}}

\authormark{E. M. Ait Ben Hassi, M. Jakhoukh, L. Maniar and W. Zouhair}

\received{Date}{0}{Year}
\revised{Date}{0}{Year}
\accepted{Date}{0}{Year}

%\editor{Associate Editor: Name}

\abstract{The primary focus of this paper is to establish the internal null controllability for the one-dimensional heat equation featuring dynamic boundary conditions. This achievement is realized by introducing a new Carleman estimate and an observability inequality for the corresponding backward system. In conclusion, the paper includes a set of numerical experiments that serve to confirm the validity of the theoretical findings and underscore the effectiveness of the designed control with a minimal $L^2$-norm.}

 %\boxedtext{
 %\begin{itemize}
%\item Key boxed text here.
% \item Key boxed text here.
% \item Key boxed text here.
% \end{itemize}}

\maketitle

\section{Introduction}\label{sec1}
The heat equation is a fundamental partial differential equation used to model the diffusion of heat or temperature in various physical systems. It describes how temperature changes over time and space due to heat conduction, making it a crucial tool in fields such as physics, engineering, and materials science. The equation considers factors like the initial temperature distribution and the behavior of temperature at the boundaries of the region under study. Through analytical or numerical solutions, it enables us to predict how temperature evolves over time and find steady-state solutions, which are valuable for designing systems that reach thermal equilibrium.

The controllability problem of the heat equation is a fascinating and challenging topic within control theory and partial differential equations. It addresses the question of whether, and to what extent, it is possible to manipulate the temperature distribution in a given domain over time by applying suitable control inputs, see for instance \cite{CGMZ'23, CGMZ'21, MOW'23}. This problem has numerous practical applications, such as designing heating or cooling systems, optimizing thermal processes, and controlling the temperature in various industrial and scientific setups. One of the key comments on the controllability of the heat equation is that it depends on the boundary conditions, which is one of the main interests of this work focusing on dynamic boundary conditions that introduces an intriguing dimension into the modeling of heat transfer phenomena. In contrast to traditional boundary conditions that are often static, dynamic boundary conditions consider time-varying factors at the domain's boundary. This concept is particularly valuable in scenarios where the temperature distribution within a system is influenced by evolving external conditions or time-dependent inputs at the boundaries. Let $\Omega \subset \R^n$, $n\geq 2$ be an open bounded domain, with a smooth boundary $\Gamma=\partial \Omega$, the dynamic boundary condition is modelized by the equation $$\partial_t y_\Gamma- \delta \Delta_\Gamma y_\Gamma + \partial_\nu y=0,$$
where $y_\Gamma=y_{\vert \Gamma}$ is the trace of the state $y$ on the boundary $\Gamma$, $\Delta_\Gamma$ is the Laplace-Beltrami and $ \partial_\nu y$ is the normal derivative with respect to the outward unit normal vector field $\nu$. We invite the interested reader to see \cite{CGMZ''23, Hintermann'89,KM'2020}.

In all the previously cited works, it has been mandatory for the diffusion term on the boundary to be present, i.e.  $\delta >0$, see \cite[Remark 3.3]{3}. This situation is notably relevant in one-dimensional scenarios. For instance the one-dimensional heat equation with dynamic boundary conditions, which can be exemplified by the cooling process of a uniformisotropic cylindrical solid bar. In this scenario, both ends of the bar are in contact with a liquid medium. The dynamic boundary conditions for this one-dimensional heat transfer model are established based on the principle of energy conservation. According to this principle, the heat dissipated by the bar through its ends equals the heat absorbed by the surrounding liquid. For a more comprehensive exposition of this concept, we refer to \cite{Langer'1932, CGMZ'22}.

Recently, in  \cite{4}, the authors have treated this case in particular situations where the coefficients do not depend neither in space nor time variables, and the control functions are of separate variables, using the moments method. For the best of our knowledge, there is no existing research that has studied controllability (theoretically or numerically) of such system using Carleman estimates. 

In this paper, we consider the following controlled heat equation with potential and dynamic boundary conditions 
\begin{empheq}[left = \empheqlbrace]{alignat=2} \label{1.1}
\begin{aligned}
&\partial_t y -  \partial_{x}(d(x) \partial_{x} y) + a(t,x)y = 1_{\omega}v(t,x) && \text{ in }\, Q_T=(0,T)\times (0,1) \\
&\partial_t y(t,j)  +d(j)(-1)^{j+1} \partial_x y(t,j) + b_j(t) y(t,j) = 0  &&\text{ on }\, (0,T),\;  j=0,1\\
&(y,y(t,0),y(t,1))_{ \vert t=0}=(y_0, \beta_0, \beta_1) &&\text{ in }\, (0,1),
\end{aligned}
\end{empheq}
where the initial data  $(y_0, \beta_0, \beta_1) \in L^2(0,1) \times \R^2$,  and the coefficients $a\in L^\infty(Q_T)$, $b_0, b_1 \in L^\infty(0,T)$. Additionally, we assume that the diffusion coefficient is a bounded function $d\in W^{1,\infty}(0,1)$ such that $d(x) \geq M \quad \text{for all} \quad x \in [0,1],$  where $M$ is a strictly positive constant. The main purpose of this article is to establish the null controllability of the above system. In other words, for a given $T>0$, $\omega \Subset (0,1)$ and initial data $(y_0, \beta_0, \beta_1) \in L^2(0,1) \times \R^2$, we look for a control  $v\in L^2((0,T)\times\omega)$  such that the corresponding solution satisfies $$\left(y(T,\cdot), y(T,0), y(T,1)  \right) =0 \quad \text{in } (0,1).$$

\section{Results}\label{sec2}
We state our main result ensuring the null controllability of the
one-dimensional heat equation with dynamic boundary conditions.
\begin{theorem}\label{thm:intro} 
For  each $T > 0$,  each nonempty open set $\omega \Subset (0,1)$ and  all initial data
$(y_0,\beta_{0}, \beta_{1})\in  L^2(0,1) \times \R^2$ there is a control $v\in L^2((0,T)\times \omega)$ 
such that the unique mild solution $\left(y , y(\cdot,0), y(\cdot,1)  \right)$ of the system \eqref{1.1} satisfies $$\left(y(T,\cdot), y(T,0), y(T,1)  \right) =0 \quad \text{in } (0,1).$$
\end{theorem}
It is currently established that the null controllability of a linear system aligns with an observability inequality for its adjoint system. To elaborate, consider \(\varphi\), the solution to the following adjoint system corresponding to \eqref{1.1}.
\begin{empheq}[left = \empheqlbrace]{alignat=2} \label{2}
\begin{aligned}
&-\partial_t \varphi - \partial_x(d(x) \partial_{x} \varphi) + a(t,x)\varphi =0 &&\text{ in }   Q_T\\
&-\partial_t \varphi(t,j)  +(-1)^{j+1} d(j)\partial_x \varphi(t,j) + b_j(t) \varphi(t,j) = 0 &&\text{ on }(0,T),\;  j=0,1\\
&(\varphi,\varphi(t,0),\varphi(t,1))_{\vert t=T}=(\varphi_T, \gamma_0, \gamma_1) &&\text{ in }(0,1).
\end{aligned}
\end{empheq}
The strategy is to establish a new Carleman estimate for the above adjoint system, leading to the following observability estimate.
\begin{equation*}
\Vert \varphi(0,\cdot)\Vert_{L^2(0,1)}^2 + \sum_{j=0,1} \vert\varphi(0,j) \vert^2   
 \leq C \int_0^T \int_{\omega} \vert \varphi \vert^2 \,dx\,dt,
\end{equation*}
and thus  the null controllability  of the system \eqref{1.1}.

The rest of the paper is organized as follows, In Section \ref{sec3}, the functional setting and some useful results associated to the heat equation with dynamic boundary conditions are presented along with the well-posedness of the corresponding systems. In Section \ref{sec4}, we elaborate a new  Carleman estimate for the system \eqref{2}. Section \ref{sec5} is devoted to establish the internal null controllability. In the final section \ref{sec6}, we present some numerical tests to illustrate our controllability results.

\section{Well-posedness of the system}\label{sec3}
In this section, we present the well-posedness of the following inhomogeneous system
\begin{empheq}[left = \empheqlbrace]{alignat=2} \label{eq:system}
\begin{aligned}
&\partial_t y - \partial_x(d(x)\partial_{x} y)+a(t,x)y = f(t,x), &&  \text{ in }  Q_T,\\
&\partial_t y(t,j)  +(-1)^{j+1}d(j) \partial_x y(t,j) + b_j(t) y(t,j) = g_j(t), && \text{on }(0,T),\;  j=0,1,\\
&(y,y(t,0),y(t,1))_{\vert t=0} = (y_0,\beta_0,\beta_1)  ,  && \text{ in } (0,1),
\end{aligned}
\end{empheq}
where $d \in W^{1,\infty}(0,1)$, $ a \in L^\infty ( Q_T)$, and $b_j \in L^\infty (0,T), j=0,1.$
For this end, let us introduce the following spaces:
\begin{align*}
&\L^2 := L^2(0,1)\times \R^2, \qquad 
\H^k := \lbrace (y,y(0),y(1)), \; y \in H^k(0,1)\rbrace\qquad \text{for \ }k\in\N,\\
&\E_1(t_0, t_1) := H^1(t_0, t_1;\L^2) \cap \L^2(t_0, t_1;\H^2) \,\text{for \ } t_1 > t_0 \in \R; \, \E_1 := \E_1(0, T).
\end{align*}
The scalar product on $\L^2$ is given by
$$\langle (u,\alpha_0,\alpha_1),(v,\beta_0,\beta_1) \rangle_{\L^2}=\langle u,v \rangle _{L^2(0,1)}+  \sum_{j=0,1} \alpha_j \beta_j  ,$$
where $u,v \in L^2(0,1)$ and $\alpha_j, \beta_j \in \R$ for j=0,1. Moreover,
$$\Vert(u,\alpha_0,\alpha_1) \Vert_{\L^2}= \big( \Vert u \Vert_{L^2(0,1)}^2 + \sum_{j=0,1} \vert \alpha_j \vert^2  \big)^{\frac{1}{2}} $$ is the associated norm to the defined scalar product.

The equation \eqref{eq:system} can be rewritten as an abstract Cauchy problem on $\L^2$ as follows

\begin{empheq}[left = \empheqlbrace]{alignat=2} \label{abstrait}
\begin{aligned}
&\mathbf{Y}^{\prime}(t)= (A + B(t)) \mathbf{Y}(t) + F(t), \qquad t \in (0,T]\\
&\mathbf{Y}(0)=\mathbf{Y}_0=\begin{pmatrix}
y_0, \beta_0, \beta_1
\end{pmatrix},
\end{aligned}
\end{empheq}
where $\mathbf{Y}(t):=\begin{pmatrix}
y(t,\cdot), y(t,0), y(t,1)
\end{pmatrix}$ and the linear operator $A: D(A) \subset \L^2 \longrightarrow \L^2$ is given by
$$A := \begin{pmatrix}
\partial_x(d(x) \partial_x) &  0 & & 0 &\\
 \\
 d(0)\partial_x \vert_{x=0}  & 0  & & 0 &\\
 \\
 -d(1)\partial_x \vert_{x=1}   & 0 & & 0 &
\end{pmatrix}, \qquad D(A)= \H^2,$$
and
$$
B(t):= \begin{pmatrix}
-a(t,\cdot) & 0 & 0 \\
\\
0 &  -b_0(t)  & 0 \\
\\
0 & 0 & - b_1(t)
\end{pmatrix}, \quad 
F(t):= \begin{pmatrix}
f(t,\cdot) \\
\\
g_1(t)\\
\\
g_2(t)
\end{pmatrix}.$$
The following first generation result is provided in  \cite[Proposition 5.2.1]{4}.
\begin{proposition}
The operator $A$ is densely defined, self-adjoint and generates a self-adjoint compact analytic $C_0$-semigroup $(e^{tA})_{t \geq 0}$ of angle $\frac{\pi}{2}$ on $\L^2$.
\end{proposition}
Next, we recall the existence and regularity results from \cite[Proposition 2.5]{3}.
\begin{proposition}
Let $f \in L^2(Q_T)$, $g_0, g_1 \in L^2(0,T),$ $(y_0, \beta_0, \beta_1) \in \L^2$ and $(\varphi_T, \gamma_0, \gamma_1) \in \L^2$. Then, the following assertions are true.
\begin{enumerate}
    \item [(a)] There is a unique mild solution $\mathbf{Y} := (y, y(\cdot,0),y(\cdot,1)) \in  C([0, T];\L^2) $ of \eqref{eq:system}. Moreover, $\mathbf{Y}$ belongs to $\E_1(\tau,T)$ and solves \eqref{eq:system} strongly on $(\tau,T)$ with initial data $\mathbf{Y}(\tau) := (y(\tau,\cdot) y(\tau,0),y(\tau,1))$, for each  $\tau \in (0, T)$. There are bounded linear operators $S(t,s)$ on $\L^2$ depending strongly continuous on $0 \leq s \leq t \leq T$ such that
  $$\mathbf{Y} (t) = S(t,0) \mathbf{Y}_0 + \int_{0}^{t} S(t,s)(f(s,\cdot), g_0(s),g_1(s)) ds; \quad t \in [0, T].$$
    \item [(b)]  Given $R > 0$, there is a constant $C = C(R) > 0$ such that for all $a$, $b_0$ and $b_1$ with
$\Vert a \Vert_{\infty}$, $\Vert b_0 \Vert_{\infty}$, $\Vert b_1 \Vert_{\infty} \leq R$ and all data the mild solution $\mathbf{Y}$ satisfies $$\Vert \mathbf{Y} \Vert_{C([0,T ],\L^2)} \leq  C ( \Vert \mathbf{Y}_0 \Vert_{\L^2} + \Vert f \Vert_{L^2( Q_T)} +\sum_{j=0,1} \Vert g_j \Vert_{L^2(0,T)}).$$
    \item [(c)] If $\mathbf{Y}_0 \in \H^1$, then the mild solution from (a) is the strong one.
    \item [(d)] The backward system of \eqref{eq:system} with $f=g_0=g_1=0$ admits a unique mild solution given by $(\varphi(t),\varphi(t,0),\varphi(t,1)) = S^*(T,t)(\varphi_T, \gamma_0,\gamma_1)$  where $S^*(T, t)$ denotes the adjoint of $S(T, t)$. It is the unique strong solution if $(\varphi_T, \gamma_0, \gamma_1) \in \H^1$.
\end{enumerate}
\end{proposition}
\section{The Carleman estimate} \label{sec4}
In this section we prove a new Carleman estimate for the following backward adjoint problem 
\begin{empheq}[left = \empheqlbrace]{alignat=2} \label{3}
\begin{aligned}
&-\partial_t \varphi - \partial_x(d(x)\partial_{x} \varphi) + a(t,x)\varphi=f(t,x) &&\text{ in }  
 Q_T,\\
&-\partial_t \varphi(t,j)  +(-1)^{j+1}d(j)\partial_x \varphi(t,j) + b_j(t) \varphi(t,j) = g_j(t) &&\text{ on }  (0,T),\;  j=0,1, \\
&(\varphi,\varphi(t,0),\varphi(t,1))_{\vert t=T}=(\varphi_T, \gamma_0,\gamma_1) &&\text{ in } (0,1).
\end{aligned}
\end{empheq}
We first need to introduce the weight functions, which are based on the following auxiliary function $\eta_0$.
\begin{lem}\label{lem}
 Given a nonempty open set $\omega' \Subset (0,1)$, there is a function $\eta_0$ that satisfies:
 \begin{align*}
     &\eta_0\in C^2(0,1), \quad \eta_0 > 0 \;\;\text{ in } (0,1), \quad \eta_0(0)=\eta_0(1)= 0, \\
     & \vert \eta'_0 \vert > 0 \;\;\text{ in }\overline{(0,1)\backslash \omega'}, \quad  \partial_\nu \eta_0(j)=(-1)^{j+1} \eta_0'(j)<0. 
\end{align*}
\end{lem}
Let $\omega'\Subset (0,1)$,  $\lambda \geq 1$,  $m>1$ and $\eta_0$ as in Lemma \ref{lem}. Following \cite{2}, we define the weight functions $\alpha$ and $\xi$ as follows:
    \begin{align}
\alpha(t,x) &= (t(T-t))^{-1} \big( e^{2\lambda m \|\eta_0\|_\infty} 
       - e^{\lambda(m\|\eta_0\|_\infty+  \eta_0(x))}\big), \notag\\
\xi(t,x) &= (t(T-t))^{-1} e^{\lambda(m\|\eta_0\|_\infty+ \eta_0(x))}\notag
\end{align}
for $(t,x) \in  (0,T) \times [0,1]$. Note that $\alpha$ and $\xi$ are $C^2(0,1)$ positive on $(0,T)\times [0,1]$, and blow up at $t=0,$ $t =T$. 
\begin{remark}:\label{rmk}
The properties of $\eta_0$ gives the following results
\begin{align}\label{eta}
\alpha(t,0)=\alpha(t,1)&:=\alpha(t), \quad \xi(t,0)=\xi(t,1):=\xi(t), \notag\\
\partial_x \alpha & = - \partial_x \xi = - \lambda \xi \eta_0',\notag\\
\partial_{xx} \alpha&  = - \lambda^2 \xi  \eta_0'^2  - \lambda \xi  \eta_0''.
\end{align}
\end{remark}
Here, we state the Carleman estimate for our adjoint system \eqref{3}.
\begin{lem}\label{2*}
Let $T>0$, $\omega$ a nonempty open subset of $(0,1)$, $a\in L^\infty( Q_T)$, and $b_0, \, b_1\in L^ \infty(0,T)$ such that, for a given $R >0$ we have $\Vert a \Vert_\infty,\, \Vert b_0 \Vert_\infty, \, \Vert b_1 \Vert_\infty\leq R$. Choose a nonempty open set $\omega'\Subset \omega$. Define $\eta_0$, $\alpha$ and $\xi$ as above with respect to $\omega'$. 
Then,  there exist constants $ C >0 $, and $\lambda_1, \, s_1 \ge 1$ such that 

\begin{align}\label{Carleman}
    &s^{-1}\int_{ Q_T} \xi^{-1} e^{- 2s \alpha} \vert \partial_t \varphi \vert^2 \, dx dt +s^{-1}\int_{ Q_T} \xi^{-1} e^{- 2s \alpha} \vert\partial_x(d(x)\partial_{x} \varphi) \vert^2 \, dx dt \notag\\
&+s^{-1}\sum_{j=0,1}\int_{0}^{T} \xi^{-1} e^{- 2s \alpha} \vert \partial_t \varphi(t,j) \vert^2 \,dt  +  s^3 \lambda^4 \int_{ Q_T} \xi^3 e^{- 2s \alpha} \vert \varphi \vert^2 \, dx \, dt \notag\\
&+ s \lambda^2 \int_{ Q_T} \xi e^{- 2s \alpha} \vert \partial_x \varphi \vert^2 \, dx \, dt +  s^3 \lambda^3 \sum_{j=0,1}\int_{0}^{T} \xi^3 e^{- 2s \alpha} \vert \varphi(t,j) \vert^2 \,dt  \notag  \\
&+ s \lambda \sum_{j=0,1} \int_{0}^{T} \xi e^{- 2s \alpha} \vert \partial_x \varphi(t,j) \vert^2\,dt \leq C \left( \int_{ Q_T} e^{-2s\alpha} \vert L \varphi \vert^2 \, dx dt \right.\\
&+ \left. \sum_{j=0,1} \int_{0}^{T} e^{-2s\alpha} \vert K\varphi(t,j)\vert^2 \,dt + s^3 \lambda^4 \int_{(0,T)\times \omega} \xi^3 e^{- 2s \alpha} \vert \varphi \vert^2 \, dx \, dt \right)\notag
\end{align}
for all $\lambda \geq \lambda_1$, $s \geq s_1$ and $(\varphi,\varphi(\cdot,0),\varphi(\cdot,1))\in \E_1$,
where 
$$L\varphi=\partial_t \varphi+ \partial_{x}(d(x) \partial_x \varphi) - a\varphi$$ and 
$$K\varphi(t,j)=\partial_t\varphi(t,j) - (-1)^{j+1} d(j)\partial_{x} \varphi(t,j) - b_{j}(t)\varphi(t,j), \qquad j=0,1.$$
\end{lem}
\begin{proof}
 We shall present the strategy of the proof step by step following some ideas in \cite{3} in a modified form. By a density argument, it suffices to consider smooth solutions $ \varphi \in C^\infty([0,T]\times [0,1])$.
We recall that the constant $C$ is generic and may change from a line to another.

\textbf{Step 1. Change of variables.}
Let $\varphi\in C^\infty([0,T]\times [0,1])$, $\lambda \ge 1$ and $s \ge 1.$ 
We will define the operators to conjugate in $(0,T)\times(0,1)$ and $(0,T) \times \{0,1\}$. For this matter, we consider the variable change $\varphi=e^{s \alpha} \psi$ and we compute $$M\psi=e^{-s\alpha} L(e^{s\alpha} \psi) \, \text{ and } \, N\psi(t,j)= e^{-s\alpha} K(e^{s\alpha} \psi(t,j)),\qquad j=0,1.$$
Adopting the same decomposition in \cite{3}, we rewrite the operators as
\begin{equation}\label{(3.7)}
 M_1\psi + M_2\psi = \tilde f \quad \text{in }Q_T, \qquad 
  N_1\psi(t,j) + N_2\psi(t,j) = \tilde g_j \quad \text{for }j=0,1,
\end{equation}
with the abbreviations
\begin{align}\label{dec1}
  M_1\psi &= - s  \lambda^2 \xi \eta_0'^2 d(x) \psi - 2s \lambda \xi \eta_0' d(x) \partial_x \psi + \partial_t\psi = \sum_{i=1,2,3}(M_1\psi)_i,\\ 
  M_2\psi  &= s^2\lambda^2 \xi^2 \eta_0'^2 d(x) \psi + \partial_{x}(d(x) \partial_x \psi) + s\psi \partial_t\alpha 
            = \sum_{k=1,2,3}(M_2\psi)_k,\notag\\
\tilde f & = e^{-s \alpha} L\varphi  + s \lambda \xi \eta_0'' d(x)\psi + s \lambda \xi \eta'_0 \partial_x d(x) \psi+ a\psi,\notag\\
N_1\psi(t,j) & =  \partial_t\psi(t,j) + s\lambda \xi(t) \partial_\nu \eta_0(j) d(j)\psi(t,j)  = \sum_{i=1,2}(N_1\psi(t,j))_i,\notag\\
N_2\psi(t,j) & =  s\psi(t,j) \partial_t\alpha(t) - (-1)^{j+1} d(j)\partial_x \psi(t,j)
     =\sum_{k=1,2}(N_2\psi(t,j))_k, \notag
\end{align}
\begin{align*}
\tilde g_{j}(t) &=e^{-s \alpha} K\varphi(t,j) + b_{j} \psi(t,j) \qquad j=0,1. 
\end{align*}
Applying the norms $\|\cdot\|_{L^2(Q_T)}^2$ (resp.\ $\|\cdot\|_{L^2(0,T)}^2$) to the equations in \eqref{(3.7)} and adding the resulting identities, we obtain
\begin{align}
\|\tilde f  \|_{L^2(Q_T)}^2 &+ \sum_{j=0,1}  \|\tilde g_j \|_{L^2(0,T)}^2 
 = \|M_1\psi\|_{L^2(Q_T)}^2 + \|M_2\psi\|_{L^2(Q_T)}^2 \notag \\
 & + \sum_{j=0,1} \|N_1\psi(\cdot,j)\|_{L^2(0,T)}^2+\sum_{j=0,1} \|N_2\psi(\cdot,j)\|_{L^2(0,T)}^2 \label{(3.8)} \\ 
  &+ 2 \sum_{i,k=1,2,3} \langle  (M_1\psi)_i, (M_2\psi)_k\rangle_{L^2(Q_T)} \notag \\
  &+ 2 \sum_{i,k=1,2} \sum_{j=0,1} \langle (N_1\psi(\cdot,j))_i, (N_2\psi(\cdot,j))_k\rangle_{L^2(0,T)}.  \notag 
\end{align}

\textbf{Step 2. Calculating the scalar product in \eqref{(3.8)}.}
\textbf{Step 2a.}  We start with the negative term
$$\langle (M_1\psi)_1, (M_2\psi)_1\rangle%_{L^2(Q_T)} 
       = - s^3\lambda^4 \int_{Q_T} \eta_0'^4 \xi^3d^2 \vert \psi \vert^2 \,d x\, d t.$$
Using integration by parts and  Remark \ref{rmk}, we further derive
\begin{align*}
 &\langle (M_1\psi)_2\, ,(M_2\psi)_1\rangle%_{ L^2(Q_T)} 
 = 3 s^3\lambda^4 \int_{Q_T} \xi^3 \eta_0'^4 d^2\vert \psi \vert^2\,d x\, d t +3 s^3\lambda^3 \int_{Q_T} \xi^3 \eta_0'^2 \eta_0'' d^2 \vert \psi \vert^2\,d x\, d t \\
&+2s^3\lambda^3 \int_{Q_T} \xi^3 \eta'^2_0 d \partial_x d \vert \psi \vert^2 \,d x\, d t - s^3\lambda^3 \sum_{j=0,1} \int_0^T \xi^3(t) (\partial_\nu \eta_0(j))^3 d(j)\vert \psi(t,j) \vert^2 \, d t.   
\end{align*}
Integrating by parts in time, we obtain
\begin{align*}
 \langle (M_1\psi)_3, (M_2\psi)_1\rangle%_{L^2(Q_T)} 
 & = \frac{1}{2} s^2\lambda^2 \int_{Q_T} \eta_0'^2 \xi^2 d\partial_t (\vert \psi \vert^2) \, dx \, dt\\
  & = - s^2\lambda^2 \int_{Q_T} \eta_0'^2  \partial_t \xi \xi d \vert \psi \vert^2 \, dx \, dt.
\end{align*}
Hence, 
\begin{align*}
&\sum_{i=1,2,3} \langle (M_1\psi)_i, (M_2\psi )_1\rangle = 2 s^3\lambda^4  \int_{Q_T} \eta_0'^4 \xi^3 d^2\vert \psi \vert^2 \,d x\, d t\notag \\
&  +3 s^3\lambda^3 \int_{Q_T} \xi^3 \eta_0'^2 \eta_0'' d^2 \vert \psi \vert^2\,d x\, d t +2s^3\lambda^3 \int_{Q_T} \xi^3 \eta'^2_0 d \partial_x d \vert \psi \vert^2 \,d x\, d t\\
&  - s^2\lambda^2 \int_{Q_T} \eta_0'^2  \partial_t \xi \xi d \vert \psi \vert^2 \, dx \, dt  - s^3\lambda^3 \sum_{j=0,1} \int_0^T \xi^3(t)(\partial_\nu \eta_0(j))^3 d(j)\vert \psi(t,j) \vert^2 \, d t.
\end{align*}

\textbf{Step 2b.} Integration by parts yields
\begin{align*}
 \langle (M_1\psi)_1, (M_2\psi)_2\rangle%_{L^2(Q_T)}
 &= - s\lambda^2 \int_{Q_T} \eta_0'^2 \xi d\psi \partial_{x}(d(x)\partial_x \psi)  \, dx\, dt \\
     &= s\lambda^2 \int_{Q_T} \eta_0'^2 \xi   d^2\vert \partial_x \psi \vert^2  \, dx\, dt  
    +2  s\lambda^2 \int_{Q_T} \xi \eta_0' \eta_0''d^2 \partial_x \psi \psi \, dx\, dt 
 \end{align*}
 \begin{align*}
 & +s\lambda^3 \int_{Q_T}\xi \eta_0'^3 d^2 \psi \partial_x \psi  \,dx\,dt +s \lambda \int_{Q_T} \xi \eta'^2_0 d \partial_x d \psi \partial_x \psi \,dx\,dt\\
&- s\lambda^2 \sum_{j=0,1} \int_0^T \eta_0'^2(j) \xi(t)d^2(j)\psi(t,j)\partial_\nu \psi(t,j)\, dt.
\end{align*}
On the other hand, we have
\begin{align*}
\langle (M_1\psi)_2, &(M_2\psi)_2\rangle%_{L^2(Q_T)} 
      = -2s\lambda \int_{Q_T}  \eta_0'  \xi  \partial_{x}(d(x)\partial_{x} \psi) d(x)\partial_x \psi \, dx \, d t \notag \\
   & = s\lambda \int_{Q_T} \xi \eta_0''d^2 \vert\partial_x \psi\vert^2 \, dx \, d t  + s\lambda^2 \int_{Q_T}  \xi  \eta_0'^2 d^2 \vert\partial_x \psi\vert^2 \, dx \, d t \\
   &\qquad-s\lambda \sum_{j=0,1} \int_{0}^{T} \xi(t) \partial_\nu \eta_0(j) d^2(j)\vert\partial_x \psi(t,j)\vert^2\, d t. \notag
\end{align*}
$\partial_x \psi$ vanishes at $t = 0$ and $t = T$. So, we obtain
\begin{align*}
 \langle (M_1\psi)_3, (M_2\psi)_2 \rangle
 &= \int_{Q_T} \partial_t \psi \partial_{x} (d(x) \partial_{x}\psi) \, dx \,d t \notag\\
 &= \sum_{j=0,1}\int_{0}^{T} \partial_t \psi(t,j) d(j) \partial_\nu \psi(t,j) \,dt.
\end{align*}
All this yield
\begin{align*}
&\sum_{i=1,2,3} \langle (M_1\psi)_i, (M_2\psi)_2\rangle = 2 s\lambda^2 \int_{Q_T} \eta_0'^2 \xi   d^2\vert \partial_x \psi \vert^2  \, dx\, dt  +s\lambda \int_{Q_T} \xi \eta_0''d^2 \vert\partial_x \psi\vert^2 \, dx \, d t 
    \notag\\
& +s\lambda^3 \int_{Q_T}\xi \eta_0'^3 d^2 \psi \partial_x \psi  \,dx\,dt +s \lambda \int_{Q_T} \xi \eta'^2_0 d \partial_x d \psi \partial_x \psi \,dx\,dt\\
&+2  s\lambda^2 \int_{Q_T} \xi \eta_0' \eta_0''d^2 \psi \partial_x \psi \, dx\, dt - s\lambda^2 \sum_{j=0,1} \int_0^T \eta_0'^2(j) \xi(t)d^2(j)\psi(t,j)\partial_\nu \psi(t,j)\, dt \notag \\
  &+\sum_{j=0,1}\int_{0}^{T} \partial_t \psi(t,j) d(j) \partial_\nu \psi(t,j) \,dt  -s\lambda \sum_{j=0,1} \int_{0}^{T} \xi(t) \partial_\nu \eta_0(j) d^2(j)\vert\partial_x \psi(t,j)\vert^2\, d t\notag.
\end{align*}

\textbf{Step 2c.} We have,
\begin{align*}
\langle (M_1\psi)_1&, (M_2\psi)_3\rangle%_{L^2(Q_T)} 
 = - s^2\lambda^2 \int_{Q_T} \eta_0'^2 \partial_t \alpha \xi d\vert \psi \vert^2 \, dx\, dt.
\end{align*}
By integrating by parts and making use of \eqref{eta}, we obtain
\begin{align*}
\langle (M_1\psi)_2, (M_2\psi)_3\rangle%_{L^2(Q_T)} 
    &= - s^2 \lambda \int_{Q_T} \partial_t \alpha \xi \eta_0' d \partial_x ( \vert \psi \vert^2) \, dx\, dt \\
    & = s^2 \lambda \int_{Q_T} \partial_x (\partial_t \alpha) \eta_0' \xi d \vert \psi \vert^2 \, dx\, dt+ s^2 \lambda^2 \int_{Q_T} \partial_t \alpha  \eta_0'^2 \xi d\vert \psi \vert^2 \, dx\, dt \\
    \end{align*}
    \begin{align*}
 & +  s^2 \lambda \int_{Q_T} \partial_t \alpha \eta_0'' \xi d \vert \psi \vert^2 \, dx\, dt + s^2 \lambda \int_{Q_T} \xi \eta'_0 \partial_t \alpha \partial_x d \vert \psi \vert^2 \, dx\, dt\\
 & - s^2 \lambda \sum_{j=0,1} \int_{0}^{T} \partial_t \alpha(t) \xi(t) \partial_\nu \eta_0(j) d(j)\vert \psi(t,j) \vert^2  \, dt. 
   \end{align*}
Integrating by parts with respect to time, we derive
\begin{align*}
\langle (M_1\psi)_3, (M_2\psi)_3\rangle%_{L^2(Q_T)} 
  &= \frac{s}{2} \int_{Q_T} \partial_t \alpha \partial_t ( \vert \psi \vert^2) \, d x \,d t 
   =  - \frac{s}{2} \int_{Q_T} \partial_{tt} \alpha \vert \psi \vert^2\, d x \,d t.
\end{align*}
This provides 
\begin{align*}
&\sum_{i=1,2,3} \langle (M_1\psi)_i, (M_2\psi)_3\rangle =s^2 \lambda  \int_{Q_T} \partial_x (\partial_t \alpha) \eta_0' \xi d \vert \psi \vert^2 \, dx\, dt\notag \\
&  +  s^2 \lambda  \int_{Q_T} \partial_t \alpha \xi \eta_0''d \vert \psi \vert^2 \, dx\, dt + s^2 \lambda \int_{Q_T} \xi \eta'_0 \partial_t \alpha \partial_x d \vert \psi \vert^2 \, dx\, dt\\
&- \frac{s}{2}  \int_{Q_T} \partial_{tt} \alpha \vert \psi \vert^2\, dx \,dt- s^2 \lambda \sum_{j=0,1} \int_{0}^{T} \partial_t \alpha(t) \xi(t)\partial_\nu \eta_0(j) d(j)\vert \psi(t,j) \vert^2  \, dt. \notag 
\end{align*}
\textbf{Step 2d.} Next, we consider the boundary terms $N_1$ and $N_2$.
\begin{align*}
\sum_{j=0,1} \langle (N_1\psi(\cdot,j))_1, (N_2\psi(\cdot,j))_1%\rangle_{L^2(0,T)} 
 &= \frac{s}{2}  \sum_{j=0,1} \int_{0}^{T} \partial_t\alpha(t) \partial_t(\vert \psi(t,j) \vert^2)\,dt, \\
    &=-\frac{s}{2} \sum_{j=0,1} \int_{0}^{T} \partial_{tt}\alpha(t) \vert \psi(t,j) \vert^2 \,dt.
\end{align*}
\begin{align*}
\sum_{j=0,1} \langle (N_1\psi(\cdot,j))_2, (N_2\psi(\cdot,j))_1\rangle%_{L^2(0,T)} 
 =  s^2\lambda \sum_{j=0,1}\int_{0}^{T} \partial_t \alpha(t) \xi(t) \partial_\nu \eta_0(j) d(j) \vert \psi(t,j) \vert^2\, dt.
\end{align*}
Finally, we have the summand
\begin{align*}
\sum_{j=0,1} \langle (N_1\psi(\cdot,j)&)_1, (N_2\psi(\cdot,j))_2\rangle%_{L^2(0,T)} 
     = -\sum_{j=0,1}\int_{0}^{T} d(j) \partial_t \psi(t,j) \partial_\nu \psi(t,j)\, dt,
\end{align*}
and
\begin{align*}
 \sum_{j=0,1} \langle (N_1\psi(\cdot,j))_2, (N_2\psi(\cdot,j))_2\rangle%_{L^2(0,T)} 
     =- s \lambda \sum_{j=0,1} \int_{0}^{T}  \partial_\nu \eta_0(j) \xi(t) d^2(j) \psi(t,j) \partial_\nu \psi(t,j) \, dt.
\end{align*}
Thus, 
\begin{align*}
&\sum_{i,k=1,2} \sum_{j=0,1} \langle (N_1\psi(\cdot,j))_i, (N_2\psi(\cdot,j))_k\rangle =  \\
& -\frac{s}{2} \sum_{j=0,1} \int_{0}^{T} \partial_{tt}\alpha(t) \vert \psi(t,j) \vert^2 \,dt  - \sum_{j=0,1}\int_{0}^{T}d(t,j) \partial_t \psi(t,j) \partial_\nu \psi(t,j)\, dt\\
\end{align*}
\begin{align*}
&+ s^2\lambda \sum_{j=0,1} \int_{0}^{T} \partial_t \alpha(t) \xi(t) \partial_\nu \eta_0(j) d(j)\vert \psi(t,j) \vert^2\, dt 
 \\& - s \lambda \sum_{j=0,1} \int_{0}^{T} \partial_\nu \eta_0(j) \xi(t) d^2(j) \psi(t,j) \partial_\nu \psi(t,j) \, dt. \notag
\end{align*}

\textbf{Step 3. Estimating the terms from below.} We gather the ultimate equalities in Steps { \bf 2a--2d}.
\begin{align}\label{total}
&\sum_{i,k=1,2,3} \langle (M_1\psi)_i, (M_2\psi)_k\rangle+\sum_{i,k=1,2} \sum_{j=0,1} \langle (N_1\psi(\cdot,j))_i, (N_2\psi(\cdot,j))_k\rangle \notag \\
&=2 s^3\lambda^4  \int_{Q_T} \eta_0'^4 \xi^3 d^2\vert \psi \vert^2 \,d x\, d t +3 s^3\lambda^3 \int_{Q_T} \xi^3 \eta_0'^2 \eta_0'' d^2 \vert \psi \vert^2\,d x\, d t \notag\\
 &+ 2 s^3 \lambda^3 \int_{Q_T} \xi^3 \eta'^2_0 d \partial_x d \vert \psi \vert^2 \, dx \, dt- s^2\lambda^2 \int_{Q_T} \eta_0'^2  \partial_t \xi \xi d \vert \psi \vert^2 \, dx \, dt \notag \\
& + s^2 \lambda \int_{Q_T} \partial_t \alpha \xi \eta_0'' d \vert \psi \vert^2 \, dx\, dt +s^2 \lambda\int_{Q_T} \xi \eta'_0 \partial_t \alpha \partial_x d \vert \psi \vert^2 \, dx\, dt\notag \\
&+s^2 \lambda \int_{Q_T} \partial_x (\partial_t \alpha) \eta_0' \xi d \vert \psi \vert^2 \, dx\, dt  - \frac{s}{2} \int_{Q_T} \partial_{tt} \alpha \vert \psi \vert^2\, d x \,d t \notag \\
 &+2s \lambda^2 \int_{Q_T} \eta_0'^2 \xi d^2 \vert \partial_x \psi \vert^2 \, dx \,d t + s \lambda \int_{Q_T} \xi \eta_0'' d^2 \vert \partial_x \psi \vert^2 \, dx \,d t \notag \\
& +2 s \lambda^2 \int_{Q_T} \xi \eta_0' \eta_0'' d^2 \psi \partial_x \psi \, dx \,d t + s \lambda^3 \int_{Q_T} \xi \eta_0'^3 d^2 \psi \partial_x \psi \, dx \,d t  \\
&+ s\lambda \int_{Q_T} \xi \eta'^2_0 d \partial_x d \psi \partial_x \psi \, dx\, dt - s^3\lambda^3 \sum_{j=0,1} \int_{0}^{T} \xi^3(t) (\partial_\nu \eta_0(j))^3 d(j) \vert \psi(t,j) \vert^2 \, dt \notag\\
&-\frac{s}{2} \sum_{j=0,1} \int_{0}^{T} \partial_{tt}\alpha(t) \vert \psi(t,j) \vert^2 \,dt - s\lambda \sum_{j=0,1}\int_{0}^{T} \xi(t) \partial_\nu \eta_0(j) d^2(j) \vert\partial_x \psi(t,j)\vert^2\, dt \notag \\
& -  s\lambda^2 \sum_{j=0,1} \int_{0}^{T}\eta_0'^2 (j) \xi(t) d^2(j)\psi(t,j)\partial_\nu \psi(t,j) \, dt \notag \\
&-s \lambda\sum_{j=0,1} \int_{0}^{T}\partial_\nu \eta_0(j) \xi(t) d^2(j)\psi(t,j)\partial_\nu \psi(t,j) \, dt =\sum_{l=1}^{13} I_l + \sum_{l=1}^{5} B_l. \notag
\end{align}
To estimate those terms, we will use the following basic pointwise estimates on $[0,1]$
\begin{align}\label{(3.9)}
 &e^{2\lambda m \Vert \eta_0 \Vert_\infty} \leq e^{2 \lambda (m \Vert \eta_0 \Vert_\infty + \eta_0(x))},\quad 
 \vert \partial_x \alpha \vert ,\, \vert \partial_x \xi \vert \leq C \lambda \xi, \quad \vert \partial_t \alpha \vert, \, \vert\partial_t \xi \vert\leq C T \xi^2\notag \\
 & \vert \partial_{tt} \alpha \vert \leq  CT^2 \xi^3,\quad \vert\partial_x(\partial_t \alpha) \vert \leq C \lambda T \xi^2 \quad \text{ and } \quad \vert \partial_{xx} \alpha\vert \leq C \lambda^2 \xi, \text{ for } \lambda \geq 1,
\end{align}
as well as the fact that $d \in W^{1,\infty}(0,1)$.
From Lemma \ref{lem}, there exists a constant $C >0 $ such that 
\begin{align*}
\sum_{l=1}^{3}I_l & \geq 2 s^3\lambda^4  \int_{(0,T)\times (0,1)\setminus \omega'} \eta_0'^4 \xi^3 d^2\vert \psi \vert^2 \,dx\, dt +3 s^3\lambda^3 \int_{Q_T} \xi^3 \eta_0'^2 \eta_0''d^2 \vert \psi \vert^2\,d x\, d t \\
&+ 2 s^3 \lambda^3 \int_{Q_T} \xi^3 \eta'^2_0 d \partial_x d \vert \psi \vert^2 \, dx \, dt
\end{align*}
\begin{align*}
&\geq  C s^3\lambda^4  \int_{(0,T)\times (0,1)\setminus\omega'}  \xi^3 \vert \psi \vert^2 \,dx\, dt - C s^3\lambda^3 \int_{Q_T} \xi^3 \vert \psi \vert^2\,dx\, dt \\
&\geq  C s^3\lambda^4  \int_{Q_T}  \xi^3 \vert \psi \vert^2 \,dx\, dt - C s^3\lambda^4  \int_{(0,T)\times \omega'}  \xi^3 \vert \psi \vert^2 \,dx\, dt \\
&- C s^3\lambda^3 \int_{Q_T} \xi^3 \vert \psi \vert^2\,dx\, dt  \\
&\geq C s^3\lambda^4  \int_{Q_T}  \xi^3 \vert \psi \vert^2 \,dx\, dt  -  C s^3\lambda^4  \int_{(0,T)\times \omega'}  \xi^3 \vert \psi \vert^2 \,dx\, dt
\end{align*}
for $\lambda \geq \lambda_0 $ and $s \geq s_0 $ large enough. We finally obtain, from Lemma \ref{lem}
\begin{align*}
\sum_{l=1}^{3}I_l+B_1 &\geq C s^3\lambda^4  \int_{Q_T}  \xi^3 \vert \psi \vert^2 \,dx\, dt  - 2 C s^3\lambda^4  \int_{(0,T)\times \omega'}  \xi^3 \vert \psi \vert^2 \,dx\, dt \notag \\ 
&+ C s^3\lambda^3 \sum_{j=0,1} \int_0^T \xi^3(t) \vert \psi(t,j) \vert^2 \, dt.
\end{align*}
From the estimations in \eqref{(3.9)}, we obtain 
\begin{align*}
\sum_{l=4}^{8}I_l+ B_2 & \geq -Cs^2\lambda^2(T+T^2) \int_{Q_T}  \xi^3 \vert \psi \vert ^2 \,dx\, dt -C s  T^2  \sum_{j=0,1} \int_{0}^{T} \xi^3(t) \vert \psi(t,j) \vert ^2\,dt
\end{align*}
for $s \geq 1$ and $\lambda \geq 1$. Using again Lemma \ref{lem}, we obtain 
\begin{align*}
I_{9}+I_{10} & \geq C s \lambda^2 \int_{Q_T}  \xi \vert \partial_x \psi \vert^2 \, dx \,dt- C s \lambda^2 \int_{(0,T)\times\omega'}  \xi d^2\vert \partial_x \psi\vert^2 \, dx \,dt.
\end{align*}
Next, we apply Young's inequality to $I_{11}$, $I_{12}$ and $I_{13}$.
\begin{align*}
\vert I_{11}\vert  %&\leq C s \int_{Q_T} \vert \lambda^2 \xi^{\frac{1}{2}} \psi \vert \vert \xi^{\frac{1}{2}} \partial_x \psi\vert \, dx \,dt, \\
&\leq C s \lambda^4 \int_{Q_T}  \xi \vert \psi\vert^2 \, dx \,dt + C s \int_{Q_T}  \xi \vert \partial_x \psi \vert^2 \, dx \,dt, 
\end{align*}
\begin{align*}
\vert I_{12}\vert %&\leq C \lambda^2 \int_{Q_T} \vert s\lambda \xi \psi \vert \vert  \partial_x \psi\vert  \, dx \,dt \\
&\leq C s^2 \lambda^4 \int_{Q_T}  \xi^2 \vert \psi\vert^2 \, dx \,dt + C \lambda^2 \int_{Q_T}  \vert \partial_x \psi\vert^2 \, dx \,dt 
\end{align*}
and
\begin{align*}
\vert I_{13}\vert 
&\leq C s^2 \int_{Q_T}  \xi^2 \vert \psi\vert^2 \, dx \,dt + C \lambda^2 \int_{Q_T}  \vert \partial_x \psi\vert^2 \, dx \,dt .
\end{align*}
It follows that
\begin{align*}
\sum_{l=11}^{13}I_{l} &\geq - C s \lambda^4 \int_{Q_T}  \xi \vert \psi\vert^2 \, dx \,dt - C s \int_{Q_T}  \xi \vert \partial_x \psi\vert^2 \, dx \,dt \\
&-C s^2 \lambda^4 \int_{Q_T}  \xi^2 \vert \psi\vert^2 \, dx \,dt - C \lambda^2 \int_{Q_T}  \vert \partial_x \psi\vert^2 \, dx \,dt.
\end{align*}
Using Young's inequality again on $B_4$ and $B_5$, we obtain for $ \varepsilon >  0 $ small enough 

\begin{align*}
\vert B_4\vert %&\leq  C s \sum_{j=0,1} \int_{0}^{T} \xi(t) \vert\lambda^{\frac{3}{2}}\psi(t,j)\vert \vert\lambda^{\frac{1}{2}}\partial_\nu \psi(t,j)\vert \, dt \\
& \leq C_\varepsilon s \lambda^3 \sum_{j=0,1} \int_{0}^{T} \xi(t) \vert\psi(t,j)\vert^2 \, dt+ C \varepsilon s \lambda \sum_{j=0,1} \int_{0}^{T} \xi(t) \vert\partial_x\psi(t,j)\vert^2 \, dt,
\end{align*}
and 
\begin{align*}
\vert B_5\vert
& \leq C_\varepsilon s \lambda  \sum_{j=0,1} \int_{0}^{T} \xi(t)\vert \psi(t,j)\vert^2 \, dt+ C \varepsilon s \lambda \sum_{j=0,1} \int_{0}^{T} \xi(t) \vert \partial_x\psi(t,j)\vert^2 \, dt.
\end{align*}
It follows that
\begin{align*}
B_4+B_5 \geq - C_\varepsilon s \lambda^3 \sum_{j=0,1} \int_{0}^{T} \xi(t)\vert\psi(t,j)\vert^2 \, dt -C \varepsilon s\lambda \sum_{j=0,1} \int_{0}^{T} \xi(t) \vert \partial_x\psi(t,j)\vert^2 \, dt
\end{align*}
for $\lambda \geq 1$. Subsequently, we estimate equality  \eqref{total} from below
\begin{align*}
&\sum_{i,k=1,2,3} \langle (M_1\psi)_i, (M_2\psi)_k\rangle+\sum_{i,k=1,2} \sum_{j=0,1} \langle (N_1\psi(\cdot,j))_i, (N_2\psi(\cdot,j))_k\rangle  \\
&\geq  C s^3 \lambda^4 \int_{Q_T} \xi^3 \vert\psi\vert^2 \, dx \, dt -C s^2 \lambda^2 (T+T^2)\int_{Q_T} \xi^3 \vert\psi\vert^2 \, dx \, dt \\
& - C s \lambda^4 \int_{Q_T} \xi \vert\psi\vert^2 \, dx \, dt- C s^2 \lambda^4 \int_{Q_T} \xi^2 \vert\psi\vert^2 \, dx \, dt \\ 
&-C s^3 \lambda^4 \int_{(0,T)\times \omega'} \xi^3 \vert\psi\vert^2 \, dx \, dt + C s \lambda^2 \int_{Q_T} \xi \vert\partial_x \psi\vert^2 \, dx \, dt \\ 
&-C s \lambda \int_{Q_T} \xi \vert\partial_x \psi\vert^2 \, dx \, dt-C \int_{Q_T} (s\xi+ \lambda^2) \vert\partial_x \psi\vert^2 \, dx \,  dt \\
&  -  C s \lambda^2 \int_{(0,T)\times \omega'} \xi d^2 \vert\partial_x \psi\vert^2 \, dx \, dt+ Cs^3 \lambda^3 \sum_{j=0,1} \int_{0}^{T} \xi^3(t) \vert\psi(t,j)\vert^2 \,dt\\
& - C s  T^2 \sum_{j=0,1} \int_{0}^{T} \xi^3(t) \vert\psi(t,j)\vert^2 \,dt  - C_\varepsilon s \lambda^3 \sum_{j=0,1} \int_{0}^{T} \xi(t) \vert\psi(t,j)\vert^2 \,dt \\ 
& + Cs \lambda  \sum_{j=0,1} \int_{0}^{T} \xi(t) \vert\partial_x \psi(t,j)\vert^2 \,dt -  C\varepsilon s \lambda \sum_{j=0,1} \int_{0}^{T} \xi(t) \vert\partial_x \psi(t,j)\vert^2 \,dt.
\end{align*}
Hence, we obtain the inequality
\begin{align*}
&\sum_{i,k=1,2,3} \langle (M_1\psi)_i, (M_2\psi)_k\rangle+\sum_{i,k=1,2} \sum_{j=0,1} \langle (N_1\psi(\cdot,j))_i, (N_2\psi(\cdot,j))_k\rangle  \\
&\geq  C s^3 \lambda^4 \int_{Q_T} \xi^3 \vert \psi \vert^2 \, dx \, dt + C s \lambda^2 \int_{Q_T} \xi \vert\partial_x \psi\vert^2 \, dx \, dt \\
& +C s^3 \lambda^3 \sum_{j=0,1} \int_{0}^{T} \xi^3(t) \vert\psi(t,j)\vert^2 \,dt + C s \lambda \sum_{j=0,1} \int_{0}^{T} \xi(t) \vert\partial_x \psi(t,j)\vert^2 \,dt  \\
&  -C s^3 \lambda^4 \int_{(0,T)\times \omega'} \xi^3 \vert\psi\vert^2 \, dx \, dt - C s \lambda^2 \int_{(0,T)\times \omega'} \xi d^2 \vert\partial_x \psi\vert^2 \, dx \, dt
\end{align*}
for $s \geq  C(T+T^2)$ and $\lambda \geq \lambda_0$ large enough.

\textbf{Step 4. The transformed estimate.} Combining the above estimation with \eqref{(3.8)}, we obtain
\begin{align*}
\|\tilde f\|_{L^2(Q_T)}^2 &+ \sum_{j=0,1}  \|\tilde g_j \|_{L^2(0,T)}^2 \geq\|M_1\psi\|_{L^2(Q_T)}^2 + \|M_2\psi\|_{L^2(Q_T)}^2 \notag \\
  & + \sum_{j=0,1} \|N_1\psi(\cdot,j)\|_{L^2(0,T)}^2 +\sum_{j=0,1} \|N_2\psi(\cdot,j)\|_{L^2(0,T)}^2 \notag \\
  &+  C s^3 \lambda^4 \int_{Q_T} \xi^3 \vert\psi\vert^2 \, dx \, dt + C s \lambda^2 \int_{Q_T} \xi \vert\partial_x \psi\vert^2 \, dx \, dt \\
& +C s^3 \lambda^3 \sum_{j=0,1} \int_{0}^{T} \xi^3(t) \vert\psi(t,j)\vert^2 \,dt + C s \lambda \sum_{j=0,1} \int_{0}^{T} \xi(t) \vert\partial_x \psi(t,j)\vert^2 \,dt  \notag \\
&  -C s^3 \lambda^4 \int_{(0,T)\times \omega'} \xi^3 \vert\psi\vert^2 \, dx \, dt - C s \lambda^2 \int_{(0,T)\times \omega'} \xi d^2 \vert\partial_x \psi\vert^2 \, dx \, dt. \notag
\end{align*}
The expressions for $\tilde f$ and $\tilde g_j$ lead to additional lower order terms which can be absorbed to the left-hand side for large $\lambda$ and $s$. Indeed, we obtain
\begin{align*}
    \|\tilde f\|_{L^2(Q_T)}^2 
    &\leq C \int_{Q_T} e^{-2s \alpha}\vert L\varphi\vert^2 \, dx dt+ C s^2 \lambda^2 \int_{Q_T}  \xi^2 \vert \psi \vert^2  \, dx dt+ C \int_{Q_T} \vert \psi\vert^2 \, dx dt
\end{align*}
and 
\begin{align*}
   \sum_{j=0,1}  \|\tilde g_j \|_{L^2(0,T)}^2 %&= \sum_{j=0,1} \int_{0}^{T} \vert\tilde g_j \vert^2 \, dt= \sum_{j=0,1} \int_{0}^{T} \vert\hat{g}_j + b_j \psi(t,j)\vert^2 \, dt \\
    &\leq C \sum_{j=0,1} \int_{0}^{T} e^{-2s \alpha}\vert K\varphi(t,j)\vert^2 \,dt+ C \sum_{j=0,1} \int_{0}^{T} \vert \psi(t,j)\vert^2 \, dt. 
\end{align*}
Thus, for $\lambda \geq \lambda_0$ and $s \geq C T^2$, we ultimately derive the following estimation.
\begin{align}\label{34}
&\|M_1\psi\|_{L^2(Q_T)}^2 + \|M_2\psi\|_{L^2(Q_T)}^2  + \sum_{j=0,1} \|N_1\psi(\cdot,j)\|_{L^2(0,T)}^2 +\sum_{j=0,1} \|N_2\psi(\cdot,j)\|_{L^2(0,T)}^2 \notag \\ 
&+ C s^3 \lambda^4 \int_{Q_T} \xi^3 \vert\psi\vert^2 \, dx \, dt +C s \lambda^2 \int_{Q_T} \xi \vert\partial_x \psi\vert^2 \, dx \, dt \\
&+ C s^3 \lambda^3 \sum_{j=0,1} \int_{0}^{T} \xi^3(t)\vert\psi(t,j)\vert^2 \,dt +C s \lambda \sum_{j=0,1} \int_{0}^{T} \xi(t) \vert\partial_x \psi(t,j)\vert^2 \,dt  \notag\\
&\leq C \int_{Q_T} e^{-2s \alpha}\vert L \varphi \vert^2 \, dx dt + C \sum_{j=0,1} \int_{0}^{T} e^{-2s \alpha}\vert K \varphi(t,j)\vert^2 \,dt  \notag \\
&+ C s^3 \lambda^4 \int_{(0,T)\times \omega'} \xi^3 \vert\psi\vert^2 \, dx \, dt +  C s \lambda^2 \int_{(0,T)\times \omega'} \xi d \vert\partial_x \psi\vert^2 \, dx \, dt. \notag 
\end{align}
On the other hand, from the decomposition \eqref{dec1}, 
We achieve
\begin{align}\label{35}
s^{-1}\int_{Q_T} \xi^{-1} \vert\partial_t \psi\vert^2 \, dx dt &\leq \|M_1\psi\|_{L^2(Q_T)}^2  +  C s\lambda^2 \int_{Q_T} \xi \vert \partial_x \psi \vert^2 \, d x\, d t \notag \\
&+ C s\lambda^4 \int_{Q_T} \xi \vert\psi\vert^2 \, d x\, d t
\end{align}
for $\lambda \geq 1$ and $s\geq C T^2$. Combining the inequalities \eqref{34} and \eqref{35}, we obtain 

\begin{align*}
&s^{-1}\int_{Q_T} \xi^{-1} \vert\partial_t \psi\vert^2 \, dx dt + \|M_2\psi\|_{L^2(Q_T)}^2   + \sum_{j=0,1} \|N_1\psi(\cdot,j)\|_{L^2(0,T)}^2   \\ 
  &+\sum_{j=0,1} \|N_2\psi(\cdot,j)\|_{L^2(0,T)}^2 + C s^3 \lambda^4 \int_{Q_T} \xi^3 \vert \psi\vert^2 \, dx \, dt + C s \lambda^2 \int_{Q_T} \xi \vert\partial_x \psi\vert^2 \, dx \, dt  \\
& +C s^3 \lambda^3 \sum_{j=0,1} \int_{0}^{T} \xi^3(t)\vert\psi(t,j)\vert^2 \,dt +C s \lambda \sum_{j=0,1} \int_{0}^{T} \xi(t) \vert\partial_x \psi(t,j)\vert^2 \,dt  \\
&\leq C \int_{Q_T} e^{-2s \alpha} \vert L \varphi \vert^2 \, dx dt + C \sum_{j=0,1} \int_{0}^{T} e^{-2s \alpha}\vert K \varphi(t,j)\vert^2 \,dt    \\
&\qquad\qquad + C s^3 \lambda^4 \int_{(0,T)\times \omega'} \xi^3 \vert\psi\vert^2 \, dx \, dt + C s \lambda^2 \int_{(0,T)\times \omega'} \xi d \vert\partial_x \psi\vert^2 \, dx \, dt. 
\end{align*} 
Similarly, one can handle the identical terms on $Q_T$ and in $(0,T) \times \{0,1 \}$, which yields to
\begin{align*}
&s^{-1}\int_{Q_T} \xi^{-1} \vert\partial_t \psi\vert^2 \, dx dt +  s^{-1}\int_{Q_T} \xi^{-1} \vert\partial_{x}(d(x) \partial_x \psi)\vert^2 \, dx dt   \\ 
 & +s^{-1} \sum_{j=0,1} \int_{0}^{T} \xi^{-1}(t) \vert \partial_t \psi(t,j)\vert^2 \,dt+s^3 \lambda^4 \int_{Q_T} \xi^3 \vert \psi \vert^2 \, dx \, dt + s \lambda^2 \int_{Q_T} \xi \vert \partial_x \psi\vert^2 \, dx \, dt  \\
& + s^3 \lambda^3 \sum_{j=0,1} \int_{0}^{T} \xi^3(t)\vert\psi(t,j)\vert^2 \,dt+ s \lambda \sum_{j=0,1}\int_{0}^{T} \xi(t) \vert\partial_x \psi(t,j)\vert^2 \,dt \\
&\leq C \int_{Q_T}e^{-2s \alpha} \vert L \varphi \vert^2 \, dx dt + C \sum_{j=0,1} \int_{0}^{T} e^{-2s \alpha} \vert K \varphi(t,j)\vert^2 \,dt \\
& +C s^3 \lambda^4 \int_{(0,T)\times \omega'} \xi^3 \vert\psi\vert^2 \, dx \, dt  +  C s \lambda^2 \int_{(0,T)\times \omega'} \xi d \vert\partial_x \psi\vert^2 \, dx \, dt, 
\end{align*}
for $\lambda \geq \lambda_0$ and $s \geq C(T+T^2)$. Also using $\omega'\Subset \omega$, we absorb the local derivative term on the right-hand side by the integral on $Q_T$ on the left-hand side. For this matter, we introduce the function 
$\rho$ defined by
$$\rho \in \mathcal{C}_c^2(\omega), \qquad \rho \equiv 1 \text{ in } \omega', \qquad 0 \leq \rho \leq 1. $$
Two integration by parts yield
\begin{align*}
&s\lambda^2 \int_{(0,T)\times \omega'}  \xi   d\vert\partial_x \psi \vert^2  \, dx\, dt  \leqslant s\lambda^2 \int_{(0,T)\times \omega} \rho  \xi   d \vert\partial_x \psi\vert^2  \, dx\, dt  \\
&= - s \lambda^2 \int_{(0,T)\times \omega} \xi\partial _x \rho d \partial_x \psi \psi  \, dx\, dt  
- s \lambda^3 \int_{(0,T)\times \omega}   \xi \rho  \eta'_0 d \partial_x \psi \psi  \, dx\, dt \\
&- s \lambda^2 \int_{(0,T)\times \omega} \rho  \xi   \partial_{x}(d(x) \partial_x \psi) \psi  \, dx\, dt.
\end{align*}
By the same manner, for $\varepsilon>0$ we obtain

\begin{align*}
    - s \lambda^2 \int_{(0,T)\times \omega} \xi\partial _x \rho d \partial_x \psi \psi  \, dx\, dt &\leq \varepsilon s \lambda^2 \int_{(0,T)\times \omega} \xi\vert   \partial_x \psi\vert^2  \, dx\, dt \\
    &+ C_{\varepsilon} s \lambda^2 \int_{(0,T)\times \omega} \xi\vert \psi\vert^2  \, dx\, dt
\end{align*}
and
\begin{align*}
    - s \lambda^2 \int_{(0,T)\times \omega}  \rho  \partial_x \xi d \partial_x \psi \psi  \, dx\, dt  &\leq \varepsilon s \lambda^2 \int_{(0,T)\times \omega} \xi \vert  \partial_x \psi\vert^2  \, dx\, dt \\
    &+ C_{\varepsilon} s \lambda^4 \int_{(0,T)\times \omega} \xi \vert\psi \vert^2  \, dx\, dt,
\end{align*}
also,
\begin{align*}
   - s \lambda^2 \int_{(0,T)\times \omega} \rho  \xi   \partial_{x}(d \partial_x \psi) \psi  \, dx\, dt  &\leq \varepsilon s^{-1} \int_{(0,T)\times \omega} \xi^{-1} \vert   \partial_{x}(d(x) \partial_x \psi) \vert^2  \, dx\, dt \\
   &+ C_{\varepsilon} s^3 \lambda^4 \int_{(0,T)\times \omega} \xi^3 \vert \psi\vert^2  \, dx\, dt.
\end{align*}
We finally obtain, for a sufficiently small $\varepsilon$

\begin{align}\label{46}
&s^{-1}\int_{Q_T} \xi^{-1} \vert\partial_t \psi\vert^2 \, dx dt + s^{-1}\int_{Q_T} \xi^{-1} \vert\partial_{x}(d(x) \partial_x \psi) \vert^2 \, dx dt \notag  \\ 
& + s^{-1} \sum_{j=0,1} \int_{0}^{T} \xi^{-1}(t) \vert\partial_t \psi(t,j)\vert^2 \,dt+ s^3 \lambda^4 \int_{Q_T} \xi^3 \vert\psi\vert^2 \, dx \, dt \notag \\ 
  &  + s \lambda^2 \int_{Q_T} \xi \vert\partial_x \psi\vert^2 \, dx \, dt+ s^3 \lambda^3 \sum_{j=0,1} \int_{0}^{T} \xi^3(t)\vert\psi(t,j)\vert^2 \,dt\\
  &+ s \lambda \sum_{j=0,1} \int_{0}^{T} \xi(t) \vert\partial_x \psi(t,j)\vert^2 \,dt \leq C \left( \int_{Q_T}e^{-2s\alpha} \vert L \varphi \vert^2 \, dx dt \right. \notag \\
& \left.+ \sum_{j=0,1} \int_{0}^{T} e^{-2s\alpha}\vert K\varphi(t,j) \vert^2 \,dt  + s^3 \lambda^4 \int_{(0,T)\times \omega} \xi^3 \vert\psi\vert^2 \, dx \, dt \right)\notag
\end{align}
for $\lambda \geq \lambda_0$ and $s \geq C(T+T^2)$. Finally, inserting   $ \psi=e^{-s \alpha} \varphi$ into \eqref{46}, we obtain
\begin{align*}
    & s^{-1}\int_{Q_T} \xi^{-1} e^{- 2s \alpha} \vert\partial_t \varphi\vert^2 \, dx dt +s^{-1}\int_{Q_T} \xi^{-1} e^{- 2s \alpha} \vert\partial_{x}(d(x) \partial_x \varphi) \vert^2 \, dx dt \\
&  
    +s^{-1}\sum_{j=0,1}\int_{0}^{T} \xi^{-1}(t) e^{- 2s \alpha(t)} \vert\partial_t \varphi(t,j)\vert^2 \,dt  +  s^3 \lambda^4 \int_{Q_T} \xi^3 e^{- 2s \alpha} \vert\varphi\vert^2 \, dx \, dt  \\
& +  s \lambda^2 \int_{Q_T} \xi e^{- 2s \alpha} \vert\partial_x \varphi\vert^2 \, dx \, dt +  s^3 \lambda^3 \sum_{j=0,1}\int_{0}^{T} \xi^3(t) e^{- 2s \alpha(t)} \vert\varphi(t,j)\vert^2 \,dt    
  \end{align*}
\begin{align*}
&+ s \lambda \sum_{j=0,1} \int_{0}^{T} \xi(t) e^{- 2s \alpha(t)} \vert\partial_x \varphi(t,j)\vert^2\,dt  \leq C \left( \int_{Q_T}e^{-2s\alpha} \vert L \varphi \vert^2 \, dx dt \right.\\
& \left. +  \sum_{j=0,1} \int_{0}^{T} e^{-2s\alpha}\vert K\varphi(t,j)\vert^2 \,dt + s^3 \lambda^4 \int_{(0,T)\times \omega} \xi^3 e^{- 2s \alpha} \vert\varphi\vert^2 \, dx \, dt \right)
\end{align*}
for $\lambda\geq \lambda_1=\max \{\lambda_0,1\}$ and $s \geq s_1=C(T+T^2)$.
\end{proof}

\section{Internal Null controllability}\label{sec5}
In this section, we utilize the Carleman estimate to establish null controllability for \eqref{1.1}. To achieve this, our initial step involves deriving an observability estimate for the backward system \eqref{2}.
\begin{proposition}
There is a constant $C>0$ such that for all $(\varphi_T, \gamma_0,\gamma_1) \in \L^2$ the mild solution $(\varphi, \varphi(\cdot,0),\varphi(\cdot,1))$ of the backward problem \eqref{2} satisfies
\begin{align*}
    \Vert \varphi(0,\cdot) \Vert^2_{L^2(0,1)}+ \sum_{j=0,1} \vert \varphi(0,j) \vert^2 \leq  C e^{(T+\frac{1}{T})} \int_{(0,T)\times \omega} \vert\varphi\vert^2 \, dx \, dt.
\end{align*}
Moreover, for $(y_0 ,\beta_0, \beta_1) \in \L^2$ the mild solution  $(y, y(\cdot,0),y(\cdot,1))$ of the forward problem \eqref{1.1} satisfies
\begin{align*}
    \Vert y(T,\cdot) \Vert^2_{L^2(0,1)}+ \sum_{j=0,1} \vert y(T,j) \vert^2 \leq   C e^{(T+\frac{1}{T})} \int_{(0,T)\times \omega} \vert y \vert^2 \, dx \, dt.
\end{align*}
\end{proposition}
\begin{proof}
% I don't what exactly is the problem
The proof mainly follows from the Carleman estimate \eqref{Carleman} with an adaptation of the same ideas in \cite{2} and \cite[Proposition 4.1]{3}. Due to the similarity of concepts, we omit the details.
\end{proof}
For the case of inhomogeneous one-dimensional heat equation, we introduce the weighted $L^2$-spaces
$$Z_{(0,1)}=\{f \in L^2(Q_T): e^{s \alpha} \xi^{-3/2} f \in L^2(Q_T) \},$$
provided with scalar product given by
$$\langle f_1, f_2 \rangle_{Z_{(0,1)}}= \int_{Q_T}  f_1 f_2 e^{-2s \alpha} \xi^{3} \, dxdt ,$$
and
$$Z_{j}=\{g_j \in : e^{s \alpha} \xi^{-3/2} g_j \in L^2(0,T) \}, $$
with the following scalar product 
$$\langle g^1_j, g^2_j \rangle_{Z_{j}}= \int_0^T  g^1_j g^2_j e^{-2s \alpha} \xi^{3} \, dt. $$
We note that those inhomogeneities are exponentially decaying at $t = 0$ and $t = T$. Then, we obtain the following result for which the proof can be derived by applying similar ideas of  \cite[Theorem 4.2]{3}. %I don't know where the comma is missing!!!!
\begin{theorem}\label{thm2} 
For  each $T > 0$,  each nonempty open set $\omega \Subset (0,1)$ and  all initial data
$(y_0,\beta_{0}, \beta_{1})\in  \L^2$, $f \in Z_{(0,1)}$ and $g_j \in Z_{j}$,  there is a control $v\in L^2((0,T)\times \omega)$ 
such that the unique mild solution $\left(y , y(\cdot,0), y(\cdot,1)  \right)$ of the system \eqref{eq:system} satisfies $$\left(y(T,\cdot), y(T,0), y(T,1)  \right) =0 \quad \text{in} \quad(0,1).$$
\end{theorem}

\section{Numerical experiments}\label{sec6}
In the following section we shall present some numerical tests to illustrate our controllability result for the following one-dimensional heat equation with potential and dynamic boundary conditions 
\begin{empheq}[left = \empheqlbrace]{alignat=2}\label{eq18}
\begin{aligned}
&\partial_{t} y(t,x)-\partial_{xx}y(t,x)+ y(t,x)  =1_{\omega}v(t,x) && \, \text { in } Q_T, \\
&\partial_{t}y(t,0) - \partial_{x} y(t,0) =0 && \, \text{on } (0, T), \\
&\partial_{t}y(t,1) + \partial_{x} y(t,1) =0 && \, \text{on } (0, T), \\
&\left(y\left(0,\cdot\right),y\left(0,0\right),y\left(0,1\right)\right)=\left(y^0,c,d\right):=\vartheta^{0} && \,\text { in } (0,1),\\
\end{aligned}
\end{empheq}
where $T > 0$ is the final time,  $\omega\Subset (0,1)$ is a nonempty open subset, $\vartheta^{0}\in \L^2$ denotes the initial data,  $a \in L^\infty(Q_T)$ and $v \in L^{2}((0,T)\times \omega)$ is the control function. 

%Without loss of generality, we assumed that, for the numerical aspect, \(d(x) = 1\), and \(b_{j}(t) = 0\).

Let $\varepsilon>0$ be fixed, we define the cost functional $ J_{T,\varepsilon}: \mathbb{L}^{2} \rightarrow \mathbb{R}$ as follows:
		\begin{equation}
			J_{T,\varepsilon}(\varphi^{0})=  \frac{1}{2}\int_{0}^{T}\int_{\omega}\varphi^2(t,x)\,\,dx\, dt  + \frac{\varepsilon}{2}\|\Phi^0\|^2_{\mathbb{L}^2} +\big\langle \vartheta^{0} , \Phi (\cdot,T)\big\rangle_{\mathbb{L}^2},
		\end{equation}
where $\Phi :=(\varphi, \varphi(\cdot,0), \varphi(\cdot,1) )$ is the solution of the following homogeneous system
\begin{empheq}[left = \empheqlbrace]{alignat=2}
\begin{aligned}
&\partial_{t} \varphi(t,x)-\partial_{xx}\varphi(t,x)+ \varphi(t,x)  =0 && \, \text { in } Q_T, \\
&\partial_{t}\varphi(t,0) - \partial_{x} \varphi(t,0) =0 && \, \text { on } (0, T), \\
&\partial_{t}\varphi(t,1) + \partial_{x} \varphi(t,1) =0 && \, \text { on } (0, T), \\
&\left(\varphi\left(0,\cdot\right),\varphi\left(0,0\right),\varphi\left(0,1\right)\right)=\left(\varphi^0,c,d\right):=\Phi^{0} && \,\text { in } (0,1).\\
\end{aligned}
\end{empheq}
The numerical method used here is designed to  compute the HUM control  ( see, \cite{ Boutaayamou2021} for a detailed description). This will be done based on a penalized HUM approach combined with a Conjugate Gradient Algorithm, see  \cite{RGJLL}.

For the numerical resolution of systems with dynamic boundary conditions, we discretize only the space domain $[0,1]$ into a uniform grid $(x_j)_{j=0}^{N_x}$ of step $\rho =\frac{1}{N_x}$, and leave the time variable continuous.  

In the sequel, we will present some numerical experiments. For convenience, the system parameters are taken as follows
$$T=0.03, \quad  N_x=25, \quad \omega=(0.2, 0.8),$$
we consider the initial datum and the potential terms as
$$\vartheta^{0}(x) = 7 (1-x) \log(1+x^2), \,\, \text{in}\,\, [0,1] .$$
Let us set the regularization parameter $\varepsilon= \frac{1}{25^2}$, and stopping parameter $e_\mathcal{J}=10^{-6}$. The algorithm stops at iteration $n=70$.

Next, we plot the controlled and the uncontrolled solutions

\begin{figure}[H] 
\centering
\includegraphics[scale=0.5]{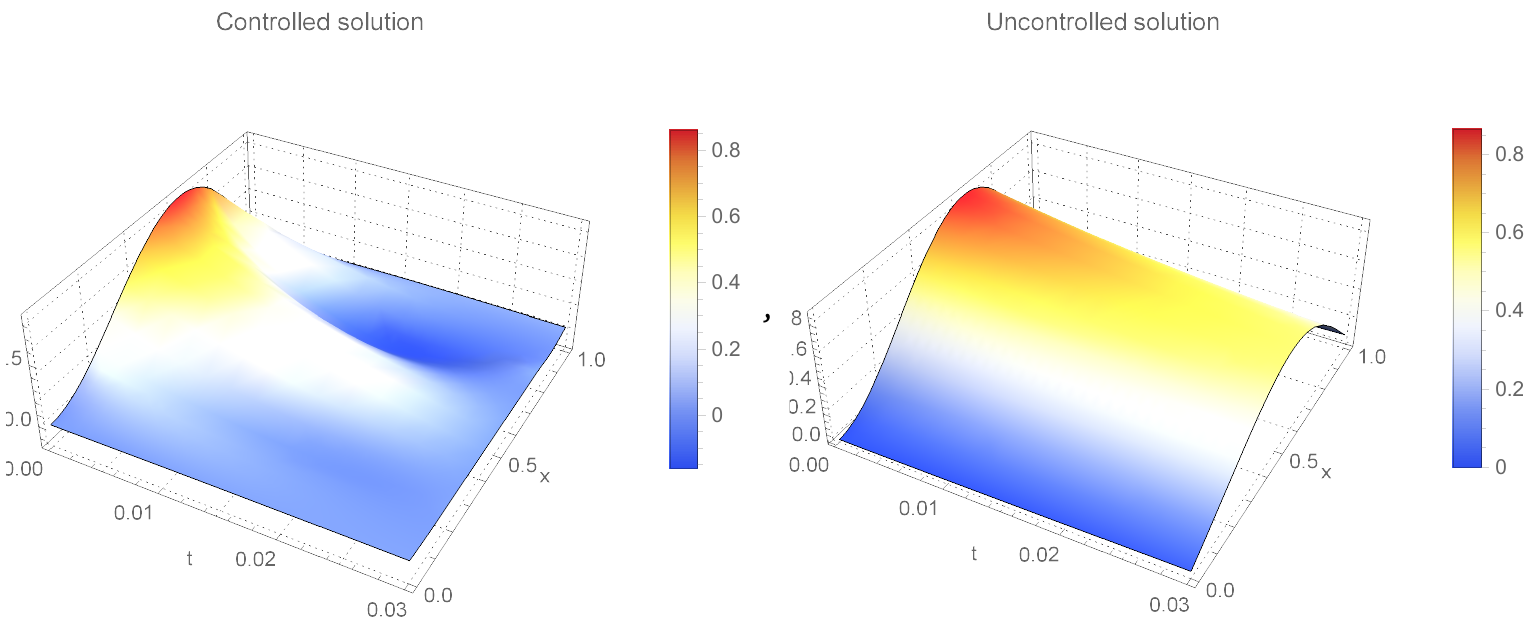}
\caption{The controlled and the uncontrolled solutions.} \label{ex1}
\end{figure}
%%%%%%%%%%%%%
\begin{figure}[H] 
\centering
\includegraphics[scale=0.5]{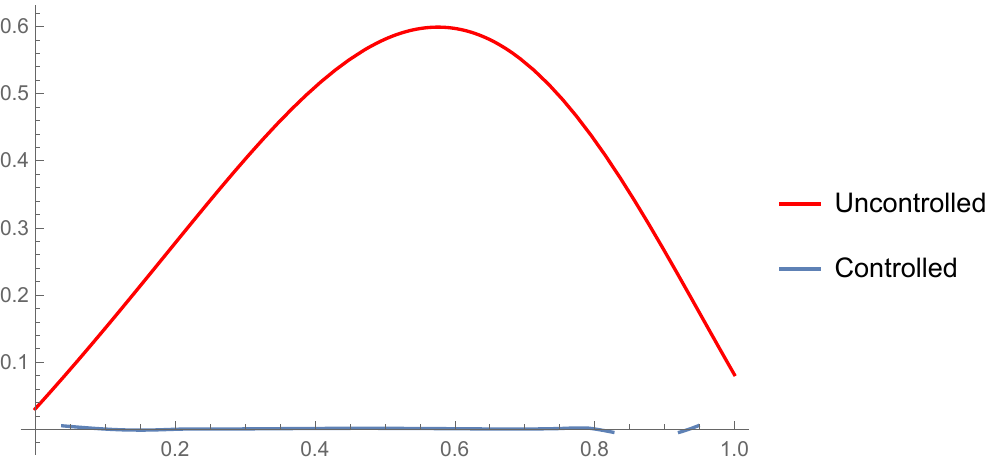}
\caption{The final state for controlled and uncontrolled solutions.} \label{ex11}
\end{figure}
%%%%%%%%%%%%%%%%%
\begin{figure}[H] 
\centering
\includegraphics[scale=0.5]{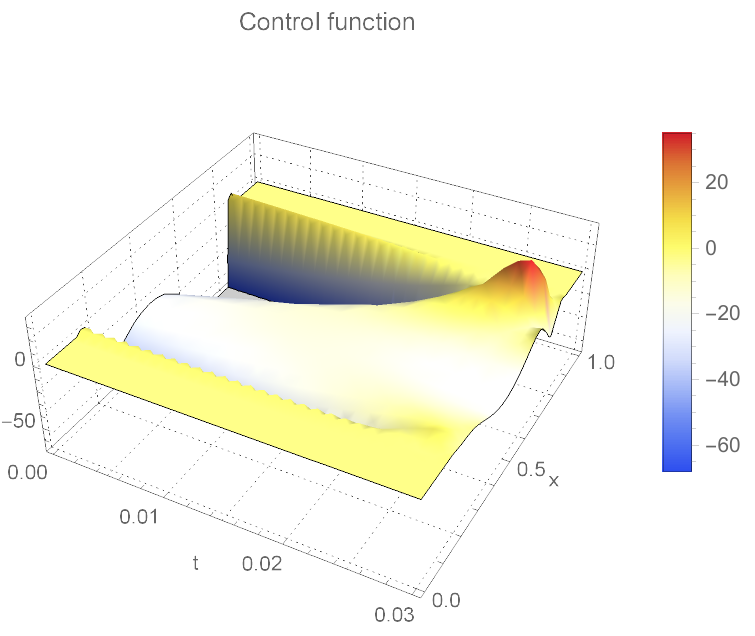}
\caption{The computed control $v$.} \label{ex111}
\end{figure}
%%%%%%%%%%%%%%%%%%%%ù
\begin{table}[ht]
\caption{Numerical results for $T=0.03$ and $e_\mathcal{J}=\frac{1}{25^3}$}. \label{t1}
\centering\setlength\arraycolsep{13pt}
\begin{tabular}{cccc}
\hline\\[-3mm]
 $\epsilon$  & $10^{-2}$ & $10^{-4}$ & $10^{-6}$ \\
\hline\\[-3mm]
 $N_{\text{iter}}$ & $15$ & $45$ & $70$ \\
\hline\\[-3mm]
 $\|\vartheta(T)\|$ & $1.226\times 10^{-2}$ & $5.402\times 10^{-3}$ & $5.397\times 10^{-3}$ \\
\hline\\[-3mm]
 $\|v\|$ & $5.007$ & $5.047$ & $5.048$ \\
\hline\\[-3mm]
\end{tabular}
\end{table}

Based on the outcomes of the numerical experiment, we notice that the minimally distributed control we've developed exhibits high efficiency. It swiftly drives the solution of the heat equation with dynamic boundary conditions toward zero. Furthermore, we observe that the solution's norm decreases as the number of iterations increases, a behavior influenced by the reduction in the stopping criterion for the algorithm. This behavior is entirely expected and normal.
\begin{remark}
To simplify the numerical aspect, we have assumed that $d(\cdot) = 1$, $b_{j}(\cdot) = 0$, and $a(\cdot,\cdot) = 1$. However, it's noteworthy that the devised algorithm continues to converge even without these assumptions, albeit requiring more iterations.
\end{remark}
\section{Conclusion and Perspectives}
This study established an observability estimate for the one-dimensional heat equation with dynamic boundary conditions, subsequently applying it to obtain the internal null controllability of the system \eqref{1.1}. The proof is based on the Carleman approach, and a numerical simulation is carried out to visually validate the theoretical findings on the internal null controllability.

In the existing literature, the controllability of semi-linear systems with static boundary conditions has been extensively studied, see for instance \cite{LMS''15, HL''14}  and the references therein. But in the context of dynamic boundary conditions, there is little work available in this area, which opens the door to possible extension. Furthermore, this work opens up possibilities for generalization, including the examination of perturbations such as nonlinearities, delays, non-local conditions. Exploring these aspects would contribute significantly to our understanding of these complex problems.

%USE THE BELOW OPTIONS IN CASE YOU NEED AUTHOR YEAR FORMAT.
%\bibliographystyle{abbrvnat}
%\bibliography{reference}

\end{document}